\documentclass{amsart}
\usepackage[all]{xy}
\usepackage{color}
\usepackage{amsthm}
\usepackage{amssymb}
\usepackage[colorlinks=true]{hyperref}

\numberwithin{equation}{section}

\newtheorem{theorem}[equation]{Theorem}
\newtheorem*{theorem*}{Theorem}
\newtheorem{lemma}[equation]{Lemma}

\newtheorem*{conjecture*}{Mamma Conjecture}
\newtheorem*{conjecture1*}{Mamma Conjecture (revisited)}
\newtheorem{proposition}[equation]{Proposition}
\newtheorem{corollary}[equation]{Corollary}
\newtheorem*{corollary*}{Corollary}

\theoremstyle{remark}

\theoremstyle{remark}
\newtheorem{remark}[equation]{Remark}

\newcommand{\cA}{{\mathcal A}}
\newcommand{\cB}{{\mathcal B}}
\newcommand{\cC}{{\mathcal C}}
\newcommand{\cD}{{\mathcal D}}

\newcommand{\cM}{{\mathcal M}}

\newcommand{\cO}{{\mathcal O}}
\newcommand{\cP}{{\mathcal P}}

\newcommand{\cT}{{\mathcal T}}
\newcommand{\cU}{{\mathcal U}}

\newcommand{\cX}{{\mathcal X}}

\newcommand{\bbC}{\mathbb{C}}

\newcommand{\bbP}{\mathbb{P}}

\newcommand{\bbQ}{\mathbb{Q}}
\newcommand{\bbZ}{\mathbb{Z}}

\DeclareMathOperator{\id}{id}

\DeclareMathOperator{\Chow}{Chow} 
\DeclareMathOperator{\DMChow}{DMChow} 

\DeclareMathOperator{\Spec}{Spec} 

\newcommand{\KM}{\mathsf{KM}}
\newcommand{\Mot}{\mathsf{Mot}}

\newcommand{\dgcat}{\mathsf{dgcat}}

\newcommand{\perf}{\mathsf{perf}}

\newcommand{\dg}{\mathsf{dg}}

\newcommand{\Hom}{\mathrm{Hom}}

\newcommand{\dgHo}{\mathsf{H}^0}

\newcommand{\op}{\mathsf{op}}

\newcommand{\too}{\longrightarrow}

\newcommand{\ie}{\textsl{i.e.}\ }

\begin{document}

\title[Exceptional collections and motivic decompositions via NC motives]{From exceptional collections \\ to motivic decompositions\\via noncommutative motives}
\author{Matilde Marcolli and Gon{\c c}alo~Tabuada}

\address{Matilde Marcolli, Mathematics Department, Mail Code 253-37, Caltech, 1200 E.~California Blvd. Pasadena, CA 91125, USA}
\email{matilde@caltech.edu} 
\urladdr{http://www.its.caltech.edu/~matilde}

\address{Gon{\c c}alo Tabuada, Department of Mathematics, MIT, Cambridge, MA 02139, USA and
Departamento de Matem{\'a}tica e CMA, FCT-UNL, Quinta da Torre, 2829-516 Caparica, Portugal}
\email{tabuada@math.mit.edu}
\urladdr{http://math.mit.edu/~tabuada}

\subjclass[2000]{13D09, 14A22, 14C15}
\date{\today}

\keywords{Noncommutative algebraic geometry, Deligne-Mumford stacks, exceptional collections, semi-orthogonal decompositions, Chow motives, noncommutative motives}

\abstract{Making use of noncommutative motives we relate exceptional collections (and more generally semi-orthogonal decompositions) to  motivic decompositions. On one hand we prove that the Chow motive $M(\cX)_\bbQ$ of every smooth and proper Deligne-Mumford stack $\cX$, whose bounded derived category $\cD^b(\cX)$ of coherent schemes admits a full exceptional collection, decomposes into a direct sum of tensor powers of the Lefschetz motive. Examples include projective spaces, quadrics, toric varieties, homogeneous spaces, Fano threefolds, and moduli spaces. On the other hand we prove that if $M(\cX)_\bbQ$ decomposes into a direct direct sum of tensor powers of the Lefschetz motive and moreover $\cD^b(\cX)$ admits a semi-orthogonal decomposition, then the noncommutative motive of each one of the pieces of the semi-orthogonal decomposition is a direct sum of $\otimes$-units. As an application we obtain a simplification of Dubrovin's conjecture.}
}

\maketitle
\vskip-\baselineskip
\vskip-\baselineskip

\bigskip

{\em Dedicated to Yuri Manin, on the occasion of his $75^{\mathrm{th}}$ birthday.}


\section*{Introduction}

Let $\cX$ be a smooth and proper Deligne-Mumford (=DM) stack~\cite{DM}. In order to study it we can proceed in two distinct directions. On one direction we can associate to $\cX$ its different Weil cohomologies $H^\ast(\cX)$ (Betti, de Rham, $l$-adic, and others; see \cite[\S8]{BM}) or more intrinsically its Chow motive $M(\cX)_\bbQ$ (with rational coefficients); see \S\ref{sec:statements}. On another direction we can associate to $\cX$ its bounded derived category $\cD^b(\cX):=\cD^b(\mathrm{Coh}(\cX))$ of coherent sheaves; see \cite{Rouquier}.

In several cases of interest (projective spaces, quadrics, toric varieties, homogeneous spaces, Fano threefolds, moduli spaces, and others; see \S\ref{sec:examples}) the derived category $\cD^b(\cX)$ admits a ``weak decomposition'' into simple pieces. The precise formulation of this notion goes under the name of a {\em full exceptional collection}; consult \cite[\S1.4]{Huy} for details. This motivates the following general questions:

\smallbreak

{\bf Question A:} {\em What can it be said about the Chow motive $M(\cX)_\bbQ$ of a smooth and proper DM stack $\cX$ whose bounded derived category $\cD^b(\cX)$ admits a full exceptional collection ? Does $M(\cX)_\bbQ$ also decomposes into simple pieces ?}

\smallbreak

{\bf Question B:} {\em Conversely, what can it be said about the bounded derived category $\cD^b(\cX)$ of a smooth and proper DM stack $\cX$ whose Chow motive $M(\cX)_\bbQ$ decomposes into simple pieces ? }

\smallbreak

In this article, making use of the theory of noncommutative motives, we provide a precise and complete answer to Question A and a partial answer to Question B.

\section{Statement of results}\label{sec:statements}
Throughout the article we will work over a perfect base field $k$. Let us denote by $\cD\cM(k)$ the category of smooth and proper DM stacks (over $\Spec(k)$) and by $\cP(k)$ its full subcategory of smooth projective varieties. Recall from \cite[\S8]{BM} \cite[\S4.1.3]{Andre} the construction of the (contravariant) functors
\begin{eqnarray*}
h(-)_\bbQ: \cD\cM(k)^\op \too \DMChow(k)_\bbQ && M(-)_\bbQ: \cP(k)^\op \too \Chow(k)_\bbQ
\end{eqnarray*}
with values in the categories of Deligne-Mumford-Chow motives and Chow motives, respectively. There is a natural commutative diagram
\begin{equation}\label{eq:diagram-one}
\xymatrix{
\cD\cM(k)^\op \ar[d]_-{h(-)_\bbQ}  & *+<2.5ex>{\cP(k)^\op} \ar@{_{(}->}[l] \ar[d]^-{M(-)_\bbQ} \\
\DMChow(k)_\bbQ & \Chow(k)_\bbQ \ar[l]^-\sim
}
\end{equation}
and as shown in \cite[Thm.~2.1]{Toen} the lower horizontal functor is a $\bbQ$-linear equivalence. By inverting it we obtain then a well-defined functor
\begin{eqnarray}\label{eq:func-final}
\cD\cM(k)^\op \too \Chow(k)_\bbQ && \cX \mapsto M(\cX)_\bbQ\,.
\end{eqnarray} 
Our first main result, which provides an answer to Question A, is the following:
\begin{theorem}\label{thm:main}
Let $\cX \in \cD\cM(k)$. Assume that $\cD^b(\cX)$ admits a full exceptional collection $(E_1, \ldots, E_m)$ of length $m$. Then, there is a choice of integers (up to permutation) $l_1, \ldots, l_m \in \{0, \ldots, \mathrm{dim}(\cX)\}$ giving rise to a canonical isomorphism
\begin{equation}\label{eq:decomposition} 
M(\cX)_\bbQ \simeq {\bf L}^{\otimes l_1} \oplus \cdots \oplus {\bf L}^{\otimes l_m}\,,
\end{equation}
where ${\bf L} \in \Chow(k)_\bbQ$ denotes the Lefschetz motive and ${\bf L}^{\otimes l}, l \geq 0$, its tensor powers (with the convention ${\bf L}^{\otimes 0}= M(\Spec(k))_\bbQ$); see \cite[\S4.1.5]{Andre}.
\end{theorem}
Intuitively speaking, Theorem~\ref{thm:main} shows that the existence of a full exceptional collection on $\cD^b(\cX)$ ``quasi-determines'' the Chow motive $M(\cX)_\bbQ$. The indeterminacy is only on the number of tensor powers of the Lefschetz motive. Note that this indeterminacy {\em cannot} be refined. For instance, the bounded derived categories of $\Spec(k) \amalg \Spec(k)$ and $\bbP^1$ admit full exceptional collections of length $2$ but the corresponding Chow motives are distinct
$$ M(\Spec(k))_\bbQ^{\oplus 2} = M(\Spec(k) \amalg \Spec(k))_\bbQ \neq M(\bbP^1)_\bbQ = M(\Spec(k))_\bbQ \oplus {\bf L}\,.$$
Hence, Theorem~\ref{thm:main} furnish us the maximum amount of data, concerning the Chow motive, that can be extracted from the existence of a full exceptional collection. 
\begin{corollary}\label{cor:main}
Let $\cX$ be a smooth and proper DM stack satisfying the conditions of Theorem~\ref{thm:main}. Then, for every Weil cohomology $H^\ast(-): \cD\cM(k)^\op \to \mathrm{VecGr}_K$ (with $K$ a field of characteristic zero) we have $H^n(\cX)=0$ for $n$ odd and $\mathrm{dim}H^n(\cX) \leq m$ for $n$ even.
\end{corollary}
Corollary~\ref{cor:main} can be used in order to prove negative results concerning the existence of a full exceptional collection. For instance, if there exists an odd number $n$ such that $H^n(\cX)\neq 0$, then the category $\cD^b(\cX)$ cannot admit a full exception collection. This is illustrated in Corollary~\ref{cor:cohomology}. Moreover, Corollary~\ref{cor:main} implies that a possible full exceptional collection on $\cD^b(\cX)$ has length always greater or equal to the maximum of the dimensions of the $K$-vector spaces $H^n(\cX)$, with $n$ even. 

\begin{remark}
After the circulation of this manuscript, Bernardara kindly informed us that an alternative proof of Theorem~\ref{thm:main} in the particular case where $\cX$ is a smooth projective {\em complex} variety can be found in \cite{Bernardara}; see also \cite{Bernardara1}. Moreover, Kuznetsov kindly explained us an alternative proof of Theorem~\ref{thm:main} following some of Orlov's ideas. Our approach is rather different and can also be viewed as a further step towards the complete understanding of the relationship between motives and noncommutative motives.
\end{remark}
Recall from \cite[\S10]{Duke}\cite[\S5]{survey} the construction of the {\em universal localizing invariant}
$$ \cU(-):\dgcat(k)\too \Mot(k)\,.$$
Roughly speaking, $\cU(-)$ is the {\em universal} functor defined on the category of dg categories and with values in a triangulated category that inverts Morita equivalences (see \S\ref{sub:Morita}), preserves filtered (homotopy) colimits, and sends short exact sequences (\ie sequences of dg categories which become exact after passage to the associated derived categories; see \cite[\S4.6]{ICM}) to distinguished triangles. Examples of localizing invariants include algebraic $K$-theory, Hochschild homology, cyclic homology (and all its variants),  and even topological cyclic homology. Because of this universal property, which is reminiscent from the theory of motives, $\Mot(k)$ is called the category of {\em noncommutative motives}. In order to simplify the exposition let us denote by ${\bf 1}$ the noncommutative motive $\cU(k)$. Our second main result, which provides a partial answer to Question B, is the following: (recall from \cite[\S1.4]{Huy} the notion of semi-orthogonal decomposition\footnote{A natural generalization of the notion of full exceptional collection.})
\begin{theorem}\label{thm:main2}
Let $\cX$ be a DM stack such that $M(\cX)_\bbQ \simeq {\bf L}^{\otimes l_1} \oplus \cdots \oplus {\bf L}^{\otimes l_m}$ (for some choice of integers $l_1, \ldots, l_m \in \{0, \ldots, \mathrm{dim}(\cX)\}$). Assume that $\cD^b(\cX)$ admits a semi-orthogonal decomposition $\langle \cC^1, \ldots, \cC^j,\ldots, \cC^p\rangle$ of length $p$ and let $\cC^j_\dg$ be the natural dg enhancement of $\cC_j$. Then, we have canonical isomorphisms
\begin{equation}\label{eq:canonical}
\cU(\cC^j_\dg)^\natural_\bbQ \simeq \underbrace{{\bf 1}_\bbQ \oplus \cdots \oplus {\bf 1}_\bbQ}_{n_j} \qquad 1\leq j\leq p
\end{equation}
in the category $\Mot(k)_\bbQ^\natural$ obtained from $\Mot(k)$ by first taking rational coefficients and then passing to the idempotent completion. Moreover, $\sum_{j=1}^p n_j=m$ and for every localizing invariant $L$ (with values in an idempotent complete $\bbQ$-linear triangulated category) the equality $L(\cC^j_\dg)=L(k)^{\oplus n_j}$ holds.
\end{theorem}
Note that Theorem~\ref{thm:main2} imposes strong conditions on the shape of a possible semi-orthogonal decomposition of $\cD^b(\cX)$, whenever $M(\cX)_\bbQ$ decomposes into a direct sum of tensor powers of the Lefschetz motive. Intuitively speaking, if a semi-orthogonal exists then each one of its pieces is rather simple from the noncommutative viewpoint.

The proofs of Theorems~\ref{thm:main} and \ref{thm:main2} rely on the ``bridge'' between Chow motives and noncommutative motives established by Kontsevich; see \cite[Thm.~1.1]{CvsNC}. In what concerns Theorem~\ref{thm:main}, we prove first that the dg enhancement $\cD_\dg^b(\cX)$ of $\cD^b(\cX)$ becomes isomorphic in the category of noncommutative motives to the direct sum of $m$ copies of the $\otimes$-unit. Then, making use of the mentioned ``bridge'' we show that all the possible lifts of this direct sum are the Chow motives of shape $\simeq {\bf L}^{\otimes l_1} \oplus \cdots \oplus {\bf L}^{\otimes l_m}$ with $l_1, \ldots, l_m \in \{0, \ldots, \mathrm{dim}(\cX)\}$. In what concerns Theorem~\ref{thm:main2}, the canonical isomorphisms \eqref{eq:canonical} follow from the mentioned ``bridge'' and from the decomposition of the noncommutative motive associated to $\cD^b_\dg(\cX)$; consult \S\ref{sub:proof} for details. 
\subsection*{Dubrovin conjecture}\label{sub:Dubrovin}
At his ICM address \cite{Dub}, Dubrovin conjectured a striking connection between Gromov-Witten invariants and derived categories of coherent sheaves. The most recent formulation of this conjecture, due to Hertling-Manin-Teleman \cite{HeMaTe}, is the following:

\medbreak

{\bf Conjecture:} {\it Let $X$ be a smooth projective complex variety. (i) The quantum cohomology of $X$ is (generically) semi-simple if and only if $X$ is Hodge-Tate (\ie its Hodge numbers $h^{p,q}(X)$ are zero for $p\neq q$) and the bounded derived category $\cD^b(X)$ admits a full exceptional collection. (ii) The Stokes matrix of the structure connection of the quantum cohomology identifies with the Gram matrix of the exceptional collection.
}
\medbreak

Thanks to the work of Bayer, Golyshev, Guzzeti, Ueda, and others (see~\cite{Bay,Gol,Guz,Ueda}), items (i)-(ii) are nowadays known to be true in the case of projective spaces (and its blow-ups) and Grassmannians, while item (i) is also know to be true for minimal Fano threefolds. Moreover, Hertling-Manin-Teleman proved that the Hodge-Tate property follows from the semi-simplicity of quantum cohomology. Making use of Theorem~\ref{thm:main} we prove that the Hodge-Tate property follows also from the existence of a full exceptional collection.

\begin{proposition}\label{prop:main}

\begin{itemize}
\item[(i)] Let $X$ be a smooth projective complex variety. If the bounded derived category $\cD^b(X)$ admits a full exceptional collection then $X$ is Hodge-Tate. 
\item[(ii)] Let $X$ be a smooth projective variety defined over a field $k$ which is embedded $\alpha:k \hookrightarrow \bbC$ into the complex numbers. If the bounded derived category $\cD^b(X)$ admits a full exceptional collection then the complex variety $X_\alpha$ (obtained by base change along $\alpha$) is Hodge-Tate.
\end{itemize}
\end{proposition}
By item (i) of Proposition~\ref{prop:main} we conclude then that the Hodge-Tate property is unnecessary in the above conjecture, and hence can be removed. As the referee kindly explained us, Proposition~\ref{prop:main} admits the following {\'e}tale version:
\begin{proposition}\label{prop:new}
Let $X$ be a smooth projective variety defined over a field $k$. Assume that its bounded derived category $\cD^b(X)$ admits a full exceptional collection. Then, for every prime number $l$ different from the characteristic of $k$, the {\'e}tale cohomology $H^\ast_{et}(X_{\overline{k}},\bbQ_l)$ (see \cite[\S3.4.1]{Andre}) is a direct sum of tensor powers of the Lefschetz motive (considered as a $\bbQ_l$-module over the Galois group $\mathrm{Gal}(\overline{k}/k)$).
\end{proposition}
\medbreak\noindent\textbf{Acknowledgments:} The authors are very grateful to Roman Bezrukavnikov and Yuri Manin for stimulating discussions and precise references. They would like also to thank Marcello Bernardara,  Alexander Kuznetsov, John Alexander Cruz Morales and Kirill Zaynullin for detailed comments on a previous draft. They are also very grateful to the anonymous referee for all his/her corrections, suggestions, and comments that greatly helped the improvement of the article. M.~Marcolli was partially supported by the NSF grants {\tt DMS-0901221}, {\tt DMS-1007207}, {\tt DMS-1201512}, and {\tt PHY-1205440}. G.~Tabuada was partially supported by the NEC award {\tt 2742738} and by the Portuguese Foundation for Science and Technology through the grant {\tt PEst-OE/MAT/UI0297/2011} (CMA).
\section{Examples of full exceptional collections}\label{sec:examples}
In this section we summarize the state of the art on the existence of full exceptional collections.
\subsection*{Projective spaces}
A full exceptional collection $(\cO(-n), \ldots, \cO(0))$ of length $n+1$ on the bounded derived category $\cD^b(\bbP^n)$ of the $n^{\mathrm{th}}$ projective space was constructed by Beilinson in \cite{Beilinson}.
\subsection*{Quadrics}
In this family of examples we assume that $k$ is of characteristic zero. Let $(V,q)$ be a non-degenerate quadratic form of dimension $n \geq 3$ and $Q_q \subset \bbP(V)$ the associated smooth projective quadric of dimension $d=n-2$. In the case where $k$ is moreover algebraically closed, Kapranov \cite{Kapranov} constructed the following full exceptional collection on the derived category $\cD^b(Q_q)$:
\begin{eqnarray*}
(\Sigma(-d), \cO(-d+1), \ldots, \cO(-1), \cO) && \mathrm{if}\,\,d \,\,\mathrm{is} \,\,\mathrm{odd}\\
(\Sigma_+(-d), \Sigma_{-}(-d), \cO(-d+1), \ldots, \cO(-1), \cO) && \mathrm{if}\,\, d \,\,\mathrm{is}\,\, \mathrm{even}\,,
\end{eqnarray*}
where $\Sigma_{\pm}$ (and $\Sigma$) denote the spinor bundles. When $k$ is {\em not} algebraically closed, Kapranov's full exceptional collection was generalized by Kuznetsov \cite{Kuznetsov} into a semi-orthogonal decomposition
\begin{equation}\label{eq:semi-orthogonal}
\langle \cD^b(Cl_0(Q_q)), \cO(-d + 1), \ldots , \cO \rangle\,,
\end{equation}
where $Cl_0(Q_q)$ denotes the even part of the Clifford algebra associated to $Q_q$.

\subsection*{Toric varieties}
Let $X$ be a projective toric variety with at most quotient singularities and $B$ an invariant $\bbQ$-divisor whose coefficients belong to the set $\{\frac{r-1}{r}; r \in \bbZ_{>0}\}$. A full exceptional collection on the bounded derived category $\cD^b(\cX)$ of the stack $\cX$ associated to the pair $(X,B)$ was constructed by Kawamata in \cite{Kawamata}.
\subsection*{Homogeneous spaces}
In a recent work~\cite{KP}, Kuznetsov-Polishchuk conjectured the following important result:

\medbreak

{\bf Conjecture:} {\it Assume that the base field $k$ is algebraically closed and of characteristic zero. Then, for every semisimple algebraic group $G$ and parabolic subgroup $P \subset G$ the bounded derived category $\cD^b(G/P)$ admits a full exceptional collection.
}

\medbreak

As explained by Kuznetsov-Polishchuk in \cite[page~3]{KP}, this conjecture is known to be true in several cases. For instance, when $G$ is a simple algebraic group of type $A$, $B$, $C$, $D$, $E_6$, $F_4$ or $G_2$ and $P$ is a certain maximal parabolic subgroup, a full exceptional collection on $\cD^b(G/P)$ has been constructed. The case of an arbitrary maximal parabolic subgroup $P$ of a simply connected simple group $G$ of type $B$, $C$ or $D$ was also treated by Kuznetsov-Polishchuk in \cite[Thm.~1.2]{KP}. 

\subsection*{Fano threefolds}
In this family of examples we assume that $k$ is algebraically closed and of characteristic zero. Fano threefolds have been classified by Iskovskih and Mori-Mukai into $105$ different deformation classes; see \cite{Isk,Isk1,MoMu}. Making use of Orlov's results, Ciolli~\cite{Cio} constructed for each one of the $59$ Fano threefolds $X$ which have vanishing odd cohomology a full exceptional collection on $\cD^b(X)$ of length equal to the rank of the even cohomology. By combining these results with Corollary~\ref{cor:main} we obtain the following characterization:
\begin{corollary}\label{cor:cohomology}
Let $X$ be a Fano threefold. Then, the derived category $\cD^b(X)$ admits a full exceptional collection if and only if the odd cohomology of $X$~vanishes.
\end{corollary}

\subsection*{Moduli spaces}

In a recent work, Manin and Smirnov \cite{MaSmi} constructed a full exceptional collection on the bounded derived category $\cD^b(\overline{\cM}_{0,n})$ of the moduli space of $n$-pointed stable curves of genus zero. This was done by an inductive blow-up procedure which combines Keel's presentation of $\overline{\cM}_{0,n}$ with Orlov's decomposition theorem. 

\section{Motivic decompositions}\label{sec:Mdecompositions}
Thanks to Theorem~\ref{thm:main} the Chow motive (with rational coefficients) of every one of the examples of \S\ref{sec:examples} decomposes into a direct sum of tensor powers of the Lefschetz motive. We would like to bring the attention of the reader to the fact that these motivic decompositions can in general be obtained using simple geometric arguments. Moreover, they have been established long before the corresponding full exceptional collections; consult for instance the work of Brosnan~\cite{Bro}, Chen-Gibney-Krashen~\cite{CheGiKra}, Chernousov-Gille-Merkurjev~\cite{CheGiMe}, Gille-Petrov-Semenov-Zainoulline~\cite{Zain}, Gorchinskiy-Guletskii~\cite{GoGu}, Karpenko~\cite{Karp}, K{\"o}ck \cite{Kock}, Manin~\cite{Manin}, Rost~\cite{Rost,Rost2}, and others.

In the case of quadrics (when $k$ is algebraically closed and of characteristic zero) the precise motivic decomposition is the following:
$$
M(Q_q)_\bbQ \simeq \left\{ \begin{array}{lcl}
M(\Spec(k))_\bbQ \oplus {\bf L} \oplus \cdots \oplus {\bf L}^{\otimes n}  &&  \text{if}  \,\,d \,\, \text{is} \,\,\text{odd} \\
M(\Spec(k))_\bbQ \oplus {\bf L} \oplus  \cdots \oplus {\bf L}^{\otimes n} \oplus {\bf L}^{\otimes(d/2)}  && \text{if}  \,\,d \,\, \text{is} \,\,\text{even} \,.
\end{array} \right.
$$
As explained by Brosnan in \cite[Remark~2.1]{Bro}, this motivic decomposition still holds over an arbitrary field $k$ of characteristic zero when $d$ is odd\footnote{The referee kindly explained us that the analogous result with $d$ even is false: a counterexample is given by the quadric in $\bbP^3$ over $\bbC(t)$ given the equation $xy+z^2=tw^2$.}. Hence, by combining the semi-orthogonal decomposition \eqref{eq:semi-orthogonal} of Kuznetsov with Theorem~\ref{thm:main2}, we conclude that the isomorphism  
\begin{equation}\label{eq:iso-new}
\cU(\cD^b_\dg(Cl_0(Q_q)))^\natural_\bbQ \simeq \cU(Cl_0(Q_q))^\natural_\bbQ \simeq {\bf 1}_\bbQ
\end{equation}
holds for every smooth projective quadric $Q_q$ of odd dimension.
\begin{remark}\label{rk:new}
The referee kindly explained us that \eqref{eq:iso-new} follows automatically from the fact that the even part of the Clifford algebra of every split odd dimensional quadric is a matrix algebra (and hence Morita equivalent to the base field).
\end{remark}

The case of Fano threefolds $X$ is quite interesting because the construction of the full exceptional collection precedes the precise description of $M(X)_\bbQ$, which was only recently obtained by Gorchinskiy-Guletskii in \cite[Thm.~5.1]{GoGu}. Concretely, we have the following isomorphism
$$ M(X)_\bbQ \simeq M(\Spec(k))_\bbQ \oplus M^1(X) \oplus {\bf L}^{\oplus b} \oplus (M^1(J)\otimes {\bf L}) \oplus ({\bf L}^{\otimes 2})^{\oplus b} \oplus M^5(X) \oplus {\bf L}^{\otimes 3}\,,$$
where $M^1(X)$ and $M^5(X)$ are the Picard and Albanese motives respectively, $b=b_2(X)=b_4(X)$ is the Betti number, and $J$ is a certain abelian variety defined over $k$, which is isogenous to the intermediate Jacobian  $J^2(X)$ if $k=\bbC$
\begin{remark}
Note that since the Lefschetz motive (and its tensor powers) have trivial odd cohomology we can conclude directly from Theorem~\ref{thm:main} and Corollary~\ref{cor:cohomology} that the derived category $\cD^b(X)$ of a Fano threefold $X$ admits a full exceptional collection if and only if $M(X)_\bbQ$ decomposes into a direct sum of tensor powers of the Lefschetz motive.
\end{remark}
\section{Preliminaries}
\subsection{Notations}\label{sub:notations}
Throughout the article we will reserve the letter $k$ for our perfect base field. The standard idempotent completion construction will be written as $(-)^\natural$.
\subsection{Dg categories}\label{sub:Morita}
A {\em differential graded (=dg) category}, over our base field $k$, is a category enriched over cochain complexes of $k$-vector spaces (morphisms sets are complexes) in such a way that composition fulfills the Leibniz rule $d(f \circ g)=d(f) \circ g + (-1)^{\mathrm{deg}(f)}f \circ d(g)$; consult Keller's ICM address~\cite{ICM}. The category of dg categories will be denoted by $\dgcat(k)$. Given a dg category $\cA$ we will write $\dgHo(\cA)$ for the associated $k$-linear category with the same objects as $\cA$ and morphisms given by $\dgHo(\cA)(x,y):=\dgHo\cA(x,y)$, where $\dgHo$ denotes the $0^{\mathrm{th}}$-cohomology. A dg category $\cA$ is called {\em pre-triangulated} if the associated category $\dgHo(\cA)$ is triangulated. Finally, a {\em Morita equivalence} is a dg functor $\cA \to \cB$ which induces an equivalence $\cD(\cA) \stackrel{\sim}{\to} \cD(\cB)$ on the associated derived categories; see \cite[\S4.6]{ICM}.
\subsection{Orbit categories}\label{sub:orbit}
Let $\cC$ be an additive symmetric monoidal category and $\cO \in \cC$ a $\otimes$-invertible object. As explained in \cite[\S7]{CvsNC}, we can then consider the {\em orbit category} $\cC\!/_{\!\!-\otimes \cO}$. It has the same objects as $\cC$ and morphisms given by
$$ \Hom_{\cC\!/_{\!\!-\otimes \cO}}(X,Y):= \bigoplus_{r \in \bbZ} \Hom_\cC(X,Y\otimes\cO^{\otimes r})\,.$$
The composition law is induced from $\cC$. Concretely, given objects $X, Y$ and $Z$ and morphisms
\begin{eqnarray*}
\underline{f}=\{f_r\}_{r\in \bbZ} \in \bigoplus_{r \in \bbZ} \Hom_{\cC}(X,Y\otimes \cO^{\otimes r}) && \underline{g}=\{g_s\}_{s\in \bbZ} \in \bigoplus_{s \in \bbZ} \Hom_{\cC}(Y,Z\otimes \cO^{\otimes s})\,,
\end{eqnarray*}
the $l^{\mathrm{th}}$-component of the composition $\underline{g} \circ \underline{f}$ is the finite sum
\begin{equation*}
\sum_r (g_{l-r} \otimes \cO^{\otimes r} )\circ f_r\,.
\end{equation*}
Under these definitions, we obtain an additive category $\cC\!/_{\!\!-\otimes \cO}$ and a canonical additive projection functor $\pi(-): \cC \to \cC\!/_{\!\!-\otimes \cO}$. Moreover, $\pi(-)$ is endowed with a natural $2$-isomorphism $\pi(-) \circ (-\otimes \cO) \stackrel{\sim}{\Rightarrow} \pi(-)$ and is $2$-universal among all such functors.
\subsection{$K_0$-motives}\label{sub:K0}
Recall from Gillet-Soul{\'e} \cite[\S5]{GiSou}\cite[\S5]{GiSou2} the construction\footnote{As explained in {\em loc. cit.}, the category $\KM(k)_\bbQ$ is originally due to Manin \cite{Manin}.} of the symmetric monoidal functor $\cP(k)^\op \to \KM(k)_\bbQ$ with values in the category of {\em $K_0$-motives} (with rational coefficients). Since $K_0$-motives are probably more familiar to the reader than noncommutative motives, we explain here the precise connection between the two.

Recall from Lunts-Orlov \cite[Thm.~2.13]{LO} that the triangulated category $\cD^b(X)$ of every smooth projective variety $X$ (or more generally of every smooth and proper Deligne-Mumford stack) admits a unique differential graded enhancement $\cD_\dg^b(X)$. In particular, we have an equivalence $\dgHo(\cD^b_\dg(X)) \simeq \cD^b(X)$ of triangulated categories. Since $X$ is regular every bounded complex of coherent sheaves is perfect (up to isomorphism) and so we have moreover a natural Morita equivalence $\cD_\dg^\perf(X) \simeq \cD_\dg^b(X)$. The assignment $X \mapsto \cD_\dg^\perf(X)$ gives then rise to a well-defined (contravariant) functor from $\cP(k)$ to $\dgcat(k)$. As explained in the proof of \cite[Thm.~1.1]{CvsNC}, there is a well-defined $\bbQ$-linear additive {\em fully faithful} symmetric monoidal functor $\theta_3$ making the following diagram commute
$$
\xymatrix{
\cP(k)^\op \ar[d] \ar[rr]^-{\cD^\perf_\dg(-)} && \dgcat(k) \ar[d]^-{\cU(-)^\natural_\bbQ} \\
\KM(k)_\bbQ \ar[rr]_-{\theta_3} && \Mot(k)_\bbQ^\natural\,.
}
$$
Intuitively speaking, the category of $K_0$-motives embeds fully faithfully into the category of noncommutative motives.
\section{Proof of Theorem~\ref{thm:main}}\label{sub:proof}
Let us denote by $\langle E_j\rangle, 1 \leq j \leq m$, the smallest triangulated subcategory of $\cD^b(\cX)$ generated by the object $E_j$. As explained in \cite[Example~1.60]{Huy} the full exceptional collection $(E_1, \ldots, E_j, \ldots, E_m)$ of length $m$ gives rise to the semi-orthogonal decomposition
\begin{equation*}
\cD^b(\cX)= \big{\langle} \langle E_1 \rangle, \ldots,  \langle E_j \rangle, \ldots, \langle E_m \rangle \big{\rangle}\,,
\end{equation*}
with $\langle E_j\rangle \simeq \cD^b(k)$ for $1 \leq j \leq m$. As explained in \S\ref{sub:K0}, the triangulated category $\cD^b(\cX)$ admits a (unique) differential graded (=dg) enhancement $\cD^b_\dg(\cX)$ such that $\dgHo(\cD^b_\dg(\cX)) \simeq \cD^b(\cX)$. Let us denote by $\langle E_j \rangle_\dg$ the dg enhancement of $\langle E_j \rangle$. Note that $\dgHo(\langle E_j\rangle_\dg) \simeq \langle E_j\rangle$ and that $\langle E_j\rangle_\dg \simeq \cD^b_\dg(k)$. Recall from \S\ref{sec:statements} that we have a well-defined universal localizing invariant
$$ \cU(-): \dgcat(k) \too \Mot(k)\,.$$
\begin{lemma}\label{lem:key}
The inclusions of dg categories $\cD_\dg^b(k) \simeq \langle E_j\rangle_\dg \hookrightarrow \cD^b_\dg(\cX), 1 \leq j \leq m$, give rise to an isomorphism
\begin{equation}\label{eq:iso-total}
\bigoplus_{j =1}^m \, \cU(\cD^b_\dg(k)) \stackrel{\sim}{\too} \cU(\cD^b_\dg(\cX))\,.
\end{equation}
\end{lemma}
\begin{proof}
For every $1 \leq i \leq m$, let $\langle E_i, \ldots, E_m\rangle$ be the full triangulated subcategory of $\cD^b(\cX)$ generated by the objects $E_i, \ldots, E_m$. Since by hypothesis $(E_1, \ldots, E_m)$ is a full exceptional collection of $\cD^b(\cX)$, we obtain the following semi-orthogonal decomposition
$$ \langle E_i , \ldots, E_m\rangle = \big{\langle} \langle E_i\rangle, \langle E_{i+1}, \ldots, E_m\rangle \big{\rangle}\,.$$

Now, let $\cA$, $\cB$ and $\cC$ be pre-triangulated dg categories (with $\cB$ and $\cC$ full dg subcategories of $\cA$) inducing a semi-orthogonal decomposition $\dgHo(\cA)=\big{\langle} \dgHo(\cB), \dgHo(\cC) \big{\rangle}$. As explained in \cite[\S12-13]{Duke}, we have then a {\em split} short exact sequence
\begin{equation}\label{eq:split}
\xymatrix{
0 \ar[r] & \cB \ar[r]_{\iota_\cB}  & \cA \ar[r] \ar@/_2ex/[l] & \cC \ar@/_2ex/[l]_-{\iota_\cC}  \ar[r] & 0\,,
}
\end{equation}
where $\iota_\cB$ (resp. $\iota_\cC$) denotes the inclusion of $\cB$ (resp. of $\cC$) on $\cA$. Consequently, \eqref{eq:split} is mapped by the universal localizing invariant $\cU(-)$ to a distinguished split triangle and so the inclusions $\iota_\cB$ and $\iota_\cC$ give rise to an isomorphism $\cU(\cB) \oplus \cU(\cC) \stackrel{\sim}{\to} \cU(\cA)$ in $\Mot(k)$. By applying this result to the dg enhancements
\begin{eqnarray*}
\cA:= \langle E_i, \cdots, E_m \rangle_\dg & \cB:= \langle E_i \rangle_\dg & \cC:=\langle E_{i+1}, \ldots, E_m\rangle_\dg
\end{eqnarray*}
we then obtain an isomorphism
\begin{equation}\label{eq:iso-key}
\cU(\cD^b_\dg(k)) \oplus \cU(\langle E_{i+1}, \ldots, E_m\rangle_\dg) \stackrel{\sim}{\too} \cU(\langle E_i, \ldots, E_m\rangle_\dg)
\end{equation}
for every $1 \leq i \leq m$. A recursive argument using \eqref{eq:iso-key} and the fact that $\cD^b_\dg(\cX)=\langle E_1, \ldots, E_m\rangle_\dg$ gives then rise to the above isomorphism \eqref{eq:iso-total}.
\end{proof}
Consider the following commutative diagram:
\begin{equation}\label{eq:diagram}
\xymatrix{
\cD\cM(k)^\op \ar@/^2pc/[rrr]^-{\cD^\perf_\dg(-)} \ar[dr]_{M(-)_\bbQ} & *+<2.5ex>{\cP(k)^\op}  \ar@{_{(}->}[l] \ar[rr]^{\cD^\perf_\dg(-)} \ar[d]^-{M(-)_\bbQ} && \dgcat(k)\ar[d]^{\cU(-)} \\
&\Chow(k)_\bbQ \ar[d]_{\pi(-)} && \Mot(k) \ar[d]^{(-)^\natural_\bbQ} \\
&\Chow(k)_\bbQ\!/_{\!\!-\otimes \bbQ(1)} \ar[rr]_-{R(-)} && \Mot(k)_\bbQ^\natural \,.
}
\end{equation}
Some explanations are in order. The lower-right rectangle is due to Kontsevich; see \cite[Thm.~1.1]{CvsNC}. The category $\Chow(k)_\bbQ\!/_{\!\!-\otimes \bbQ(1)}$ is the orbit category associated to the Tate motive $\bbQ(1)$ (which is the $\otimes$-inverse of the Lefschetz motive ${\bf L}$) and $R(-)$ is a $\bbQ$-linear additive fully-faithful functor; recall from \S\ref{sub:K0} since $\cX$ is regular we have a natural Morita equivalence $\cD^\perf_\dg(\cX) \simeq \cD^b_\dg(\cX)$. Finally, the upper-left triangle in the above diagram is the one associated to \eqref{eq:diagram-one}-\eqref{eq:func-final}.

The commutativity of the above diagram \eqref{eq:diagram}, the natural Morita equivalence $\cD^b_\dg(k)\simeq \cD^b_\dg(\Spec(k))$ of dg categories, and the fact that the functor $R(-)$ is additive and fully faithful, imply that the image of \eqref{eq:iso-total} under $(-)^\natural_\bbQ$ can be identified with the isomorphism
\begin{equation}\label{eq:sum}
\bigoplus_{j=1}^m \, \pi(M(\Spec(k))_\bbQ) \stackrel{\sim}{\too} \pi(M(\cX)_\bbQ)
\end{equation}
in the orbit category $\Chow(k)_\bbQ\!/_{\!\!-\otimes \bbQ(1)}$. Hence, since $M(\mathrm{Spec}(k))_\bbQ$ is the $\otimes$-unit of $\Chow(k)_\bbQ$ and the automorphism $-\otimes \bbQ(1)$ of $\Chow(k)_\bbQ$ is additive, there are morphisms
$$
 \underline{f} =\{f_r\}_{r \in \bbZ} \in  \bigoplus_{r \in \bbZ} \Hom_{\Chow(k)_\bbQ}(M(\cX)_\bbQ, \bigoplus_{j=1}^m \bbQ(1)^{\otimes r})$$
 and
 $$ \underline{g} =\{g_s\}_{s \in \bbZ}  \in  \bigoplus_{s \in \bbZ} \Hom_{\Chow(k)_\bbQ}(\bigoplus_{j=1}^m M(\Spec(k))_\bbQ ,M(\cX)_\bbQ \otimes \bbQ(1)^{\otimes s})$$
verifying the equalities $\underline{g} \circ \underline{f} = \id$ and $\underline{f} \circ \underline{g} = \id$. The equivalence of categories $\Chow(k)_\bbQ\simeq \DMChow(k)_\bbQ$, combined with the construction of the category of Deligne-Mumford-Chow motives (see \cite[\S8]{BM}), implies that
$$
\Hom_{\Chow(k)_\bbQ}(M(\cX)_\bbQ,\bigoplus_{j=1}^m \bbQ(1)^{\otimes r}) \simeq \bigoplus_{j=1}^m A^{\mathrm{dim}(\cX) +r}(\cX \times \Spec(k)) $$
and that
$$
\Hom_{\Chow(k)_\bbQ}(\bigoplus_{j=1}^m M(\Spec(k))_\bbQ,M(\cX)_\bbQ \otimes \bbQ(1)^{\otimes s})  \simeq  \bigoplus_{j=1}^m  A^s(\Spec(k)\times \cX)\,,$$
where $A^\ast(-)$ denotes the rational Chow ring of DM stacks defined by Vistoli in \cite{Vistoli}. Hence, we conclude that $f_r=0$ for $r\neq \{- \mathrm{dim}(\cX), \ldots, 0\}$ and that $g_s=0$ for $s\neq \{0, \ldots ,\mathrm{dim}(\cX)\}$. The sets of morphisms 
\begin{eqnarray*}
\{f_{-l}\,|\,  0 \leq l \leq \mathrm{dim}(\cX) \}&\mathrm{and}&\{g_l \otimes \bbQ(1)^{\otimes(-l)} \,|\, 0 \leq l \leq \mathrm{dim}(\cX)\}
\end{eqnarray*}
 give then rise to well-defined morphisms
\begin{eqnarray*}
\Phi: M(\cX)_\bbQ  \to  \bigoplus_{l=0}^{\mathrm{dim}(\cX)} \bigoplus_{j=1}^m \bbQ(1)^{\otimes (-l)} && \Psi: \bigoplus_{l=0}^{\mathrm{dim}(\cX)} \bigoplus_{j=1}^m \bbQ(1)^{\otimes(-l)}  \to  M(\cX)_\bbQ
\end{eqnarray*}
in $\Chow(k)_\bbQ$. The composition $\Psi \circ \Phi$ agrees with the $0^{\mathrm{th}}$-component of the composition $\underline{g} \circ \underline{f}=\id_{\pi(M(\cX)_\bbQ)}$, \ie it agrees with $\id_{M(\cX)_\bbQ}$. Since $\bbQ(1)^{\otimes (-l)}={\bf L}^{\otimes l}$ we conclude then that $M(\cX)_\bbQ$ is a direct summand of the Chow motive $\bigoplus_{l=0}^{\mathrm{dim}(\cX)} \bigoplus_{j=1}^{m}{\bf L}^{\otimes l}$. By definition of the Lefschetz motive we have the following equalities
$$ \Hom_{\Chow(k)_\bbQ}({\bf L}^{\otimes p}, {\bf L}^{\otimes q}) = \delta_{pq} \cdot \bbQ \qquad p,q\geq 0\,,$$
where $\delta_{pq}$ stands for the Kronecker symbol. As a consequence, $M(\cX)_\bbQ$ is in fact isomorphic to a subsume of $\bigoplus_{l=0}^{\mathrm{dim}(X)} \bigoplus_{j=1}^{m}{\bf L}^{\otimes l}$ indexed by a certain subset $S$ of $\{0, \ldots,\mathrm{dim}(X)\} \times \{1, \ldots, m\}$. By construction of the orbit category we have natural isomorphisms
$$ \pi({\bf L}^{\otimes l}) \stackrel{\sim}{\too} \pi(M(\Spec(k))_\bbQ) \qquad l \geq 0\,.$$
Hence, since the direct sum in the left-hand-side of \eqref{eq:sum} contains $m$ terms we conclude that the cardinality of $S$ is also $m$. This means that there is a choice of integers (up to permutation) $l_1, \ldots, l_m \in \{0, \ldots, \mathrm{dim}(\cX)\}$ giving rise to a canonical isomorphism
$$ M(\cX)_\bbQ \simeq {\bf L}^{\otimes l_1} \oplus \cdots \oplus {\bf L}^{\otimes l_m}$$
in $\Chow(k)_\bbQ$. The proof is then achieved.
\section{Proof of Corollary~\ref{cor:main}}
Since by hypothesis $K$ is a field of characteristic zero, the universal property of the category $\Chow(k)_\bbQ$ of Chow motives with rational coefficients (see \cite[Prop.~4.2.5.1]{Andre} with $F=\bbQ$) furnish us a (unique) additive symmetric monoidal functor $\overline{H^\ast}(-)$ making the right-hand-side triangle of the following diagram
$$
\xymatrix{
 \cD\cM(k)^\op \ar[dr]_-{M(-)_\bbQ}  & *+<2.5ex>{\cP(k)^\op}  \ar@{_{(}->}[l] \ar[d] \ar[r]^{H^\ast(-)} & \mathrm{VecGr}_K \\
& \Chow(k)_\bbQ \ar[ur]_{\overline{H^\ast}(-)} &
}
$$
commutative. Note that the commutativity of the left-hand-side triangle holds by construction. By hypothesis $\cD^b(\cX)$ admits a full exceptional collection of length $m$ and so by Theorem~\ref{thm:main} there is a choice of integers (up to permutation) $l_1, \ldots, l_m \in \{0, \ldots, \mathrm{dim}(\cX)\}$ giving rise to a canonical isomorphism
$$ M(\cX)_\bbQ \simeq {\bf L}^{\otimes l_1} \oplus \cdots \oplus {\bf L}^{\otimes l_m}\,.$$
Since the functor $\overline{H^\ast}(-)$ is additive and symmetric monoidal and $H^\ast(\cX)=\overline{H^\ast}(M(\cX)_\bbQ)$ we obtain then the following identification
$$ H^\ast(\cX) \simeq \overline{H^\ast}({\bf L})^{\otimes l_1} \oplus \cdots \oplus \overline{H^\ast}({\bf L})^{\otimes l_m}\,.$$
As proved in \cite[Prop.~4.2.5.1]{Andre} we have 
\begin{equation*}
\overline{H^n}({\bf L})\simeq \left\{ \begin{array}{cc} K& n=2 \\
 0 & n \neq 2 \end{array} \right.
\end{equation*}
and so we conclude that $H^n(\cX)=0$ for $n$ odd and that $\mathrm{dim}H^n(\cX)\leq m$ for $n$ even.
\section{Proof of Theorem~\ref{thm:main2}}
As the proof of Lemma~\ref{lem:key} shows, we can replace $\langle E_j\rangle_\dg$ by the dg category $\cC_\dg^j \subset \cD_\dg^b(\cX)$ and hence obtain the isomorphism $\bigoplus_{j=1}^p \cU(\cC^j_\dg) \simeq \cU(\cD^b_\dg(\cX))$ in $\Mot(k)$. Since by hypothesis $M(\cX)_\bbQ \simeq {\bf L}^{\otimes l_1} \oplus \cdots \oplus {\bf L}^{\otimes l_m}$ one concludes then from the above commutative diagram \eqref{eq:diagram} and from the fact that $R(-)$ is fully-faithful that $\bigoplus^p_{j=1} \,\cU(\cC^j_\dg)^\natural_\bbQ$ is isomorphic to the direct sum ${\bf 1}_\bbQ \oplus \cdots \oplus {\bf 1}_\bbQ$ (with $m$-terms). As a consequence, for every $1\leq j\leq p$, there exists a well-defined non-negative integer $n_j$ such that
\begin{eqnarray}\label{eq:isom-j}
\cU(\cC^j_\dg)^\natural_\bbQ \simeq  \underbrace{{\bf 1}_\bbQ \oplus \cdots \oplus {\bf 1}_\bbQ}_{n_j} &\mathrm{and}& \sum^p_{j=1}n_j = m\,.
\end{eqnarray}
Now, let $L: \dgcat(k) \to \cT$ be a localizing invariant. Since by hypothesis $\cT$ is an idempotent complete $\bbQ$-linear triangulated category, we obtain from the universal property of $\cU(-)$ (see \cite[Thm.~10.5]{Duke}) an induced $\bbQ$-linear triangulated functor
$$ \overline{L}: \Mot(k)^\natural_\bbQ \too \cT$$
such that $\overline{L}(\cU(\cA)^\natural_\bbQ)= L(\cA)$ for every dg category $\cA$. By the above isomorphism \eqref{eq:isom-j} we conclude then that $L(\cC^j_\dg)\simeq L(k)^{\oplus n_j} $ as claimed. This concludes the proof.

\section{Proof of Proposition~\ref{prop:main}}
Let us start by proving item (i). Since by hypothesis $\cD^b(X)$ admits a full exceptional collection, there is by Theorem~\ref{thm:main} a choice of integers (up to permutation) $l_1, \ldots, l_m \in \{0, \ldots, \mathrm{dim}(X)\}$ giving rise to a canonical isomorphism
\begin{equation}\label{eq:mot-decomp}
M(X)_\bbQ\simeq {\bf L}^{\otimes l_1} \oplus \cdots \oplus {\bf L}^{\otimes l_m} \in \Chow(\bbC)_\bbQ\,.
\end{equation}
We now proceed as in the proof of Corollary~\ref{cor:main}. Let us denote by ${\mathcal R}_H$ the realization functor from $\Chow(\bbC)_\bbQ$ to $\bbQ$-Hodge structures; see \cite[\S V.2.3]{Levine}. Thanks to \eqref{eq:mot-decomp} we obtain the isomorphism $ {\mathcal R}_H(M(X)_\bbQ ) \simeq {\mathcal R}_H({\bf L})^{\otimes l_1} \oplus \cdots \oplus {\mathcal R}_H ({\bf L})^{\otimes l_m}$. Since 
\begin{equation*}
h^{p,q}({\mathcal R}_H({\bf L})) = \left\{ \begin{array}{cc} 1& p=q \\
 0 & p \neq q \end{array} \right.
\end{equation*}
we conclude then that the Hodge numbers $h^{p,q}(X)$ are zero for  $p\neq q$, \ie that $X$ is
Hodge-Tate. This concludes the proof of item (i).

Let us now prove item (ii). As in the proof of item (i), there is a choice of integers (up to permutation) $l_1, \ldots, l_m \in \{0, \ldots, \mathrm{dim}(X)\}$ giving rise to an isomorphism
\begin{equation}\label{eq:mot-decomp1}
M(X)_\bbQ\simeq {\bf L}^{\otimes l_1} \oplus \cdots \oplus {\bf L}^{\otimes l_m} \in \Chow(k)_\bbQ\,.
\end{equation}
The field embedding $\alpha:k \hookrightarrow \bbC$ gives rise to to a $\bbQ$-linear additive  symmetric monoidal base change functor
\begin{eqnarray*}
\Chow(k)_\bbQ \too \Chow(\bbC)_\bbQ && M(X)_\bbQ \mapsto M(X_\alpha)_\bbQ
\end{eqnarray*}
which preserves the Lefschetz motive; see \cite[\S4.2.3]{Andre}. Consequently, using \eqref{eq:mot-decomp1} we obtain the canonical isomorphism $M(X)_\bbQ\simeq {\bf L}^{\otimes l_1} \oplus \cdots \oplus {\bf L}^{\otimes l_m}$ in $\Chow(\bbC)_\bbQ$. Now, using the arguments of item (i), we conclude that $X_\alpha$ is Hodge-Tate and so the proof is finished.

\section{Proof of Proposition~\ref{prop:new}}
The proof is similar to the one of Proposition~\ref{prop:main}; simply use the realization functor to {\'e}tale cohomology instead of the realization functor to $\bbQ$-Hodge structure.

\end{document}